%% file: main.tex
\documentclass[12pt]{amsart}
\usepackage{hyperref}
\usepackage{amsmath}
\usepackage{cleveref}

\usepackage{amsthm}
\usepackage{mathrsfs}
\usepackage{amssymb}
\usepackage{tikz-cd}
\usepackage{comment}
\usepackage{mathtools}
\usepackage{manfnt}
\calclayout
\usepackage{xcolor}
\newcommand{\Gr}{\mathrm{Gr}}
\newcommand{\Fl}{F}
\newcommand{\sym}{S}
\newcommand{\Bl}{Bl}
\newcommand{\PP}{\mathbb{P}}

\newcommand{\OO}{\mathcal O}

\newcommand*{\shom}{\mathcal{H}\kern -.5pt om} 
\newcommand{\rk}{\mathrm{rk}}
\newcommand{\Hom}{\mathrm{Hom}}
\newcommand{\OI}{\mathcal I}
\newcommand{\ra}{\rightarrow}

\newcommand{\twopartdef}[4]
{
	\left\{
		\begin{array}{ll}
			#1 & \mbox{if } #2 \\
			#3 & \mbox{if } #4
		\end{array}
	\right.
}
\newtheorem{theorem}{Theorem}[section]

\newtheorem{lemma}[theorem]{Lemma}
\newtheorem{proposition}[theorem]{Proposition}

\theoremstyle{definition}

\newtheorem{remark}[theorem]{Remark}
\newtheorem{example}[theorem]{Example}

\newtheorem*{theorem*}{Theorem}
\newtheorem{definition}[theorem]{Definition}

\title{Very free rational curves in Fano varieties}
\author[I. Coskun]{Izzet Coskun}
\author[G. Smith]{Geoffrey Smith}
\address{Department of Mathematics, Stat. and CS \\University of Illinois at Chicago, Chicago, IL 60607}
\email{icoskun@uic.edu, geoff@uic.edu}
\begin{document}
\thanks{During the preparation of this article the first author was partially supported by the NSF FRG grant DMS 1664296.}
\begin{abstract}
Let $X$ be a projective variety and let $C$ be a rational normal curve on $X$. We compute the normal bundle of $C$ in a general complete intersection of hypersurfaces of sufficiently large degree in $X$. As a result, we establish the separable rational connectedness of a large class of varieties, including general Fano complete intersections of hypersurfaces of degree at least three in flag varieties, in arbitrary characteristic. In addition, we give a new way of computing the normal bundle of certain rational curves in products of varieties in terms of their restricted tangent bundles and normal bundles on each factor. 
\end{abstract}
\maketitle

\input{intro}
\input{mainProof}
\input{examples}
\input{products}
\appendix

\bibliographystyle{plain}

\end{document}

%% file: intro.tex
\section{Introduction and statement of results}\label{introduction}
Spaces of rational curves on a proper variety $X$ play a fundamental role in the birational geometry and arithmetic of $X$. Given a rational curve $C$ on $X$, the normal bundle $N_{C|X}$ controls the deformations of $C$ in $X$ and carries essential information about the local structure of the space of rational curves. Consequently, the normal bundles of rational curves have been studied extensively when $X$ is $\PP^n$ (\cite{AlzatiRe,Conduche, CoskunRiedl,EisenbudVandeven, EisenbudVandeven2,GhioneSacchiero, Ran, Sacchiero, Sacchiero2}) and more generally (see for example \cite{Bridges, CR19, Kol96, LT19, Shen2}). 

In this paper, we study the normal bundle of rational curves in certain complete intersections in homogeneous varieties with the  goal of showing  the separable rational connectedness of the general such complete intersection. We work over an algebraically closed field $k$ of arbitrary characteristic. 

A variety $X$ is {\em separably rationally connected} ({\em SRC})  if there exists a variety $Y$ and a morphism $e:Y\times \PP^1\ra X$ such that the induced morphism on products,
\[
e^{(2)}: Y\times \PP^1\times \PP^1 \ra X\times X,
\]
is dominant and smooth. We refer the reader to \cite{Kol96} for a  discussion of the properties of SRC varieties. 
By the Birkhoff-Grothendieck theorem,  every vector bundle on $\PP^1$ is a direct sum of line bundles. Hence, the normal bundle of a smooth rational curve $C$ on $X$ can be written as  $N_{C\vert X}\cong \bigoplus_{1\leq i\leq \dim(X)-1}\OO(a_i)$. The curve $C$ is called {\em very free} if $N_{C\vert X}$ is ample or equivalently every $a_i$ is positive. 
The bundle $N_{C\vert X}$ is called \emph{balanced} if $\lvert a_i-a_j\rvert\leq 1$ for all $i,j$.
If $X$ is a smooth variety over an algebraically closed field, then $X$ is SRC if $X$ contains a  very free rational curve \cite[Theorem IV.3.7]{Kol96}.

In characteristic 0, rationally connected varieties, and in particular smooth Fano varieties,  are SRC \cite[Theorem V.2.13]{Kol96}.  Koll\'{a}r points out that SRC is the suitable generalization of rational connectedness to arbitrary characteristic and poses the question whether every smooth Fano variety in positive characteristic is SRC? Koll\'{a}r's question has been answered affirmatively for general Fano complete intersections in $\PP^n$  \cite{CZ14,Tia15}. The paper \cite{CR19} gives sharp bounds on the degree of very free rational curves on general Fano complete intersections in $\PP^n$. In a more negative direction, certain special Fano hypersurfaces are known not to have  very free curves of low degree \cite{Bridges, She12}.

In this paper, we give further examples of SRC varieties in positive characteristic. In the case of Grassmannians, our result reads as follows.

\begin{theorem}\label{grTheorem}
Let $d_1 , \dots ,  d_c \geq 3$ be integers.  If $\sum_{i=1}^c d_i < n$, then a general complete intersection 
  $Y = \bigcap_{i=1}^c Y_i$  of  hypersurfaces $Y_i$ of degree $d_i$ in the Grassmannian $G(k,n)$ is SRC.
\end{theorem}

\begin{remark}
More generally, let $C$ be a general rational normal curve of degree $e$ in $G(k,n)$ in its Pl\"{u}cker embedding. Let  $Y_i \subset G(k,n)$ be general hypersurfaces of degree $d_i \geq 3$ containing $C$ and let $Y = \bigcap _{i=1}^c Y_i$. Then $N_{C|Y}$ is balanced. This remains true regardless of whether $Y$ is Fano.

Observe that $Y$ is Fano precisely when $n > \sum_{i=1}^c d_i$. Hence, if $$e \left(n - \sum_{i=1}^c d_i\right) > k(n-k)-c,$$ then $C$ is a very free rational curve on $Y$. This gives the optimal degree bound for a very free rational curve on such a Fano complete intersection.
  \end{remark}

Similar statements hold for flag varieties and products.
\begin{theorem}\label{someMoreSRCVarieties}
\begin{enumerate}

\item Let $X$ be a flag variety. Let $H$ be the minimal ample divisor on $X$, and let $D_1,\ldots,D_c$ be divisor classes such that for each $i$, $D_i-3H$ is effective. Let $Y$ be a complete intersection of general hypersurfaces $Y_1,\ldots,Y_c$ of classes $D_1,\ldots,D_c$. If $-K_X-D_1-\cdots-D_c$ is ample, then $Y$ is SRC.

\item Let $X$ be a product of projective spaces. For each $1\leq i \leq c$, let $D_i$ be a divisor class of degree at least 3 on each factor space. Let $Y$ be the general complete intersection of hypersurfaces of type $D_1, \ldots, D_c$. If $-K_X-D_1-\cdots -D_c$ is ample, then $Y$ is SRC.
\end{enumerate}
\end{theorem}
\begin{remark}
A similar result holds for any homogeneous space---indeed, any Schubert variety---on which one can find a very free rational normal curve in the smooth locus.  See Theorem \ref{schubertVarieties} and Proposition \ref{wps} for further examples.
\end{remark}

To prove Theorems \ref{grTheorem} and \ref{someMoreSRCVarieties}, we construct very free rational curves on these complete intersections. We consider a general rational normal curve $C$ in $X$ and show that the normal bundle of $C$ in a general complete intersection $Y$ containing $C$ is balanced. In particular, if the complete intersection is Fano and the degree of $C$ is sufficiently large, then $C$ is a very free rational curve on $Y$. Our main technical result is the following.

\begin{theorem}\label{mainApplied}
Let $X\subset \PP^n$ be a linearly normal Cohen-Macaulay projective variety of dimension $m$ whose ideal is generated in degree $k$. Let $C$ be a rational normal curve of degree $e$ in $\PP^n$ contained in the smooth locus of $X$. Assume $C$ is very free in $X$. Let $H$ denote the hyperplane class in $\PP^n$. Fix some integer $c\leq m-2$. For each $1\leq i\leq c$, let $D_i=d_iH +E_i$ be a Cartier divisor class on $X$ with $d_i\geq \max(k,3)$ and $E_i$ an effective Cartier divisor class such that
\[
H^0(X,\OO(E_i))\ra H^0(C, \OO(E_i)\vert_C)
\]
is a surjection and each divisor class $(d_i-3)H+E_i$ is base point free. Let $Y$ be the zero locus of a general section of $\bigoplus_{i=1}^c \OO(D_i)$.  If
\[
C\cdot \left(-K_X-\sum_{1\leq i \leq c}D_i\right)\geq m-c+1,
\]
 then $Y$ has a very free rational curve and is SRC.
\end{theorem}
\noindent Moreover, the normal bundle of a rational curve $C$ in $X$ determines its normal bundle  in a general complete intersection $Y$ containing it. We make this precise in Theorem \ref{main}.

In Section \ref{products}, we discuss the normal bundle of rational curves in products.
Given a map $f:\PP^1\ra X$, let $N_f$ be the vector bundle determined by the exact sequence
\[
0\ra T\PP^1\ra f^*TX\ra N_f\ra 0.
\]
\begin{theorem}\label{productsMain}
For $1\leq i\leq r$, let  $f_i:\PP^1\ra X_i$ be an immersion into a variety $X_i$, which is smooth along the image of $f_i$. Set $X=X_1\times \cdots \times X_r$ and let $f: \PP^1 \to X$ be the map induced by the maps $f_i$. Suppose the characteristic $p$ of the base field $k$ is zero or there exists $i$ such that $$H^0(\PP^1,f_i^*(T^*X_i)(p+2))\ra H^0(\PP^1,T^*\PP^1(p+2))$$ is surjective. 
Then there exists a deformation $g$ of  $f$ such that $$g^*(TX) \cong f^*(TX) \quad \mbox{and} $$
\[
h^0(\PP^1, N^*_g(d))=\max(h^0(\PP^1,f^*(T^*X(d)))-d+1, \ \sum_{i=1}^r h^0(\PP^1,N_{f_i}^*(d))).
\]
\end{theorem}

The splitting type of a vector bundle on $\PP^1$ is determined by the cohomology of its twists. Hence, Theorem \ref{productsMain} determines the splitting type of $N_g$ in terms of $N_{f_i}$ and $f_i^*TX_i$. In general it is not possible to determine the $N_f$  in terms of $N_{f_i}$ and $f_i^*TX_i$; hence taking a deformation is crucial for our argument.
The deformation we use involves pre-composing the maps $f_i$ with automorphisms $\alpha_i$ of $\PP^1$. 

\begin{example}
Let $X_1$ and $X_2$ be smooth threefolds and for $i\in \{1,2\}$ let $f_i:\PP^1\ra X_i$ be immersions. Suppose $f_1^*(TX_1)\cong f_2^*(TX_2)\cong \OO(5)^{\oplus 2}\oplus \OO(6)$ and $N_{f_1}\cong N_{f_2}\cong \OO(6)\oplus \OO(8)$. Based on this information, the normal bundle of $C=f_1\times f_2(\PP^1)\subset X_1\times X_2$ in $X_1\times X_2$ could be as unbalanced as  
$\OO(5)^{\oplus 2}\oplus \OO(6)^{\oplus 2}\oplus \OO(8)$, but after a deformation will be $\OO(6)^{\oplus 5}$.
\end{example}
\begin{remark}
The bundle $N_{g}$ in Theorem \ref{productsMain} will not always be balanced because the normal bundles of $f_i(\PP^1)$ in each of the factor spaces $X_i$ need not have similar degrees. For instance, if $f_1:\PP^1\ra X_1$ and $f_2:\PP^1\ra  X_2$ are embeddings of smooth rational curves in smooth surfaces with self-intersection $d_1,d_2$ respectively, and $d_1,d_2<0$, then the normal bundle of the diagonal map $f:\PP^1\ra X_1\times X_2$ will be $\OO(2)\oplus \OO(d_1)\oplus \OO(d_2)$. Any deformation of $f(\PP^1)$ will also have this normal bundle.
\end{remark}

\subsection*{Organization of the paper} In \cref{hypersurfaceSections}, we prove Theorem \ref{mainApplied}. In \cref{sec-examples}, we apply Theorem \ref{mainApplied} to complete intersections in Grassmannians, flag varieties and some weighted projective spaces. In Section \cref{products}, we prove Theorem \ref{productsMain}.

\subsection*{Acknowledgments} We would like to thank Matheus Michalek, Eric Riedl, Bernd Sturmfels and Kevin Tucker for invaluable conversations about the subject matter of this paper.

%% file: mainProof.tex
\section{Complete intersections in SRC varieties}\label{hypersurfaceSections}
\subsection{Normal bundles of rational curves in complete intersections}
In this section, we will prove Theorem \ref{mainApplied}. Theorem \ref{mainApplied} is a consequence of the following result, which allows one to control the normal bundles of rational curves in certain complete intersections.
\begin{theorem}\label{main}
Let $X \subset \PP^n$ be a projective variety whose ideal sheaf is generated in degree $k$. Let $C$ be a rational normal curve of degree $e$ contained in the smooth locus of $X$. Let $c\leq \dim(X)-2$ be an integer. For $1 \leq i \leq c$, let $D_i = d_i H + E_i$ be Cartier divisor classes on $X$, where $d_i \geq \max(k,3)$, $H$ is the hyperplane class and $E_i$ are effective divisors such that the restriction map
\[
H^0(X,E_i)\ra H^0(C,E_i\vert_C)
\]
is surjective. Given a surjective map $$q\in \mathrm{Hom}\left(N_{C\vert X}, \bigoplus_{1\leq i \leq c} \OO(D_i)\vert_C\right),$$ there are hypersurfaces $Y_i$ with $[Y_i]=D_i$ such that if $Y=\bigcap_{i=1}^c Y_i$, then $Y$ is smooth of codimension $c$ along $C$ and $N_{C\vert Y} \cong \ker q$. 
\end{theorem}
To prove Theorem \ref{main}, we will use the following property of the divisor classes $D_i$.
\begin{definition}\label{def:N-surjective}
Let $X$ be a projective variety and let $C$ be a smooth curve contained in the smooth locus of $X$. For any Cartier divisor class $D$ on $X$, the exact sequence 
$$0 \to \OI_{C \vert X}^2 \to \OI_{C \vert X} \to N^*_{C\vert X} \to 0$$ induces a map $H^0(\OI_{C\vert X}(D))\ra H^0(N^*_{C\vert X}(D))$. If this map is surjective, we say $D$ is \emph{$N_{C\vert X}^*$-surjective}.
\end{definition}

We first prove two lemmas that show that certain divisors are $N_{C\vert X}^*$-surjective.

\begin{lemma}\label{lemma1}
Let $X\subset \PP^n$ be projective variety and let $C$ be a smooth curve contained in the smooth locus of $X$. Let $H$ denote the hyperplane class in $\PP^n$. Let $d$ be an integer such that $dH$ is $N^*_{C\vert \PP^n}$-surjective and $H^1(N^*_{X\vert \PP^n}(dH)\vert_C)=0$. Then $dH\vert_X$ is $N_{C\vert X}^*$-surjective.
\end{lemma}
\begin{proof}
Since $H^1(N^*_{X\vert \PP^n}(dH)\vert_C)=0$, the natural map $H^0(N^*_{C\vert \PP^n}(dH))\ra H^0(N^*_{C\vert X}(dH))$ is surjective. Since $dH$ is $N^*_{C\vert \PP^n}$-surjective, the copmposition $$H^0(\OI_{C\vert \PP^n}(dH))\ra  H^0(N^*_{C\vert X}(dH))$$ is also surjective. The lemma follows from the commutativity of the square
\[
\begin{tikzcd}
H^0(\OI_{C\vert \PP^n}(dH))\arrow[r]\arrow[d] &H^0(\OI_{C\vert X}(dH)) \arrow[d]\\
H^0(N^*_{C\vert \PP^n}(dH))\arrow[r]& H^0(N^*_{C\vert X}(dH)).
\end{tikzcd}
\]
\end{proof}

\begin{lemma}\label{lemma2}
Let $C$ be a smooth rational curve contained in the smooth locus of a proper variety $X$. Let $D$ and $E$ be Cartier divisor classes on $X$ with $E\cdot C\geq 0$ such  that 
\begin{enumerate}
    \item $D$ is $N_{C\vert X}^*$-surjective;
    \item $H^0(\OO_X(E))\ra H^0(\OO_C(E))$ is surjective.
    \item $N^*_{C\vert X}(D)$ is globally generated.
\end{enumerate}
Then $D+E$ is $N_{C\vert X}^*$-surjective.
\end{lemma}
\begin{proof}
Consider the commutative square
\[
\begin{tikzcd}
H^0(\OI_{C\vert X}(D))\otimes H^0(\OO_X(E))\arrow[r]\arrow[d]& H^0(\OI_{C\vert X}(D+E))\arrow[d]\\
H^0(N^*_{C\vert X}(D))\otimes H^0(\OO_C(E))\arrow[r]&H^0(N^*_{C\vert X}(D+E))
\end{tikzcd}
\]
The vertical map on the left-hand side is surjective by the first two hypotheses. Since $N^*_{C\vert X}(D)$ is globally generated, we have an exact sequence  $$0\to M \to  H^0(N^*_{C\vert X}(D))\otimes \OO_C \to N^*_{C\vert X}(D) \to 0,$$ where $M$ is the kernel of the natural evaluation map. By construction $H^0(M)=0$ and the long exact sequence of cohomology implies that $H^1(M)=0$. Since $C$ is a rational curve and $E \cdot C \geq 0$, we have $H^1(M(E))=0$.  Twisting the sequence by $\OO_C(E)$ and taking cohomology, we see that the bottom horizontal map is surjective.  Hence, the composite map $H^0(\OI_{C\vert X}(D))\otimes H^0(\OO_X(E))\ra H^0(N^*_{C\vert X}(D+E)) $ is surjective. Therefore, the right vertical map $H^0(\OI_{C\vert X}(D+E))\ra H^0(N^*_{C\vert X}(D+E))$  is surjective as well.
\end{proof}

The following is a special case of a theorem of Rathmann. For completeness, we provide the proof.

\begin{theorem}[Theorem 3.1, \cite{Rat16}]\label{rathmann}
Let $C$ be a rational normal curve of degree $e$ in $\PP^n$. For $b\geq 1$, the map
\[
H^0(\mathcal{I}_{C\vert \PP^n}(2H))\otimes H^0(C,\OO_C(b))\ra H^0(N^*_{C\vert \PP^n}(2e+b))
\]
is surjective.
\end{theorem}

\begin{proof}
Let $C$ be the rational normal curve embedded via the map $$(s,t) \mapsto (s^e,s^{e-1}t,\ldots,t^e,0,\ldots,0).$$ The ideal of $C$ is generated by the quadrics $f_{i,j}:=x_ix_j-x_{i+1}x_{j-1}$ for $1 \leq i < j \leq e$ and the linear forms $x_{e+1},\ldots,x_n$. Moreover, $$N^*_{C\vert \PP^n} \cong \OO(-e)^{\oplus n-e}\oplus \OO(-e-2)^{\oplus e-1},$$ where the terms in the former summand come from the sections
$dx_i\in H^0(N^*_{C\vert \PP^n}(e))$ ($e<i\leq n$) 
, and the terms in the latter summand can be chosen to be  sections $q_i\in H^0(N^*_{C\vert \PP^N}(e+2))$ given by 
\[
q_i=s^2dx_{i+2}+t^2 dx_i-2st dx_{i+1}
\] 
with $0\leq i\leq e-2$. 
A basis of global sections of $N^*_{C\vert \PP^n}(2e+b)$ then consists of sections $s^kt^\ell q_i$ with $k+\ell=e+b-2$ and sections $s^kt^\ell dx_i$ with $i>e$ and $k+\ell=e+b$. If $k\geq b-1$, the section $s^kt^\ell q_i$ is the image of $f_{i,\ell+1}\otimes s^{b-1}t-f_{i+1,\ell+1}\otimes s^b$, and if $k<b-1$, $s^kt^\ell q_i$ is the image of $f_{i,\ell-b+2}\otimes t^b-f_{i+1,\ell-b+2}\otimes st^{b-1}$. Likewise, given $i>e$, the section $s^kt^\ell dx_i$ is the image of $ x_\ell x_i\otimes s^b$ if $k\geq b$, or $x_{\ell-b}x_i\otimes t^b$ if $k=0$. Hence, the map is surjective.
\end{proof}

\begin{proposition}\label{specifiedNormalBundle}
Let $f:C\ra X$ be an immersion of a smooth rational curve $C$ in a variety $X$ smooth along $C$. Let $D_1$, \ldots, $D_c$  Cartier divisor classes on $X$  that are $N_{C\vert X}^*$-surjective. Then, given a surjection $$q:N_{C\vert X}\ra \bigoplus_{1\leq i\leq c} \OO_C(D_i)$$ there exist divisors $Y_i$ containing $C$ with class $D_i$ such that $Y:=\bigcap_{1\leq i \leq c} Y_i$ is smooth of codimension $c$ in $X$ along $C$ and the inclusion  $N_{C\vert Y}\ra N_{C\vert X}$ is the kernel of $q$. 
\end{proposition}
\begin{proof}
The map $q$ is equivalently a global section of $\bigoplus_{1\leq i\leq c} N^*_{C\vert X}(D_i)$.
Since each $D_i$ is $N^*_{C\vert X}$-surjective, there is some $s\in H^0(C,\bigoplus_{1\leq i\leq c} \mathcal{I}_{C\vert X}(D_i))$ such that $q$ is the image of $s$ under the natural map
\[
H^0(C,\bigoplus_{1\leq i\leq c} \mathcal{I}_{C\vert X}(D_i))\ra H^0(C,\bigoplus_{1\leq i\leq c} N^*_{C\vert X}(D_i)).
\]
The section $s$ induces global sections $s_i$ of each $\mathcal{I}_{C\vert X}(D_i)$. Set $Y_i=V(s_i)$. 
Then each $Y_i$ contains $C$, and has class $D_i$. Moreover, the sections $s_i$ give canonical isomorphisms $h_i:\OO_{Y_i} (D_i)\vert_C\ra N_{Y_i\vert X}\vert_C$, and the natural map
\[
q':N_{C\vert X}\ra \bigoplus_{1\leq i\leq c} N_{Y_i\vert X}\vert_C
\]
is $h\circ q$, where $h: \bigoplus_{1\leq i\leq c} \OO_C(D_i) \ra \bigoplus_{1\leq i\leq c} N_{Y_i\vert X}\vert_C$  is the direct sum of the $h_i$.
The map $q'$ induces a surjection of vector bundles,
\[
q'':TX\vert_C\ra \bigoplus_{1\leq i\leq c} N_{Y_i\vert X}\vert_C.
\]
Set $Y=\bigcap_{1\leq i \leq c} Y_i$.
Since $q''$ is surjective, it does not drop rank at any point of $C$; in particular, the fiber of $\ker(q'')$ at a point $p\in C$ consists of all $v\in T_pX$ that are contained in each of the tangent spaces $T_pY_i$.
Hence, $TY\vert_C\cong \ker (q'')$, and $TY\vert_C$ is a vector bundle of rank $\dim(X)-c$. Therefore, $Y$ is smooth of codimension $c$ along $C$, and we have that $TY\vert_C\ra TX\vert_C$ is the kernel of $q''$. 
The induced map of normal bundles $N_{C\vert Y}\ra N_{C\vert X}$ is the kernel of $q$.
\end{proof}

\begin{remark}
Whether $\ker q$ is balanced for general $Y_i$ depends on the starting bundles. For example,  the general kernel of $\OO\oplus \OO(2)^{\oplus 2}\ra \OO(2)$  is $\OO\oplus \OO(2)$ and not balanced. If the starting normal bundle is balanced and the assumptions of Proposition \ref{specifiedNormalBundle} hold, then the normal bundle of the complete intersection will stay balanced. 
\end{remark}

\begin{proof}[Proof of Theorem \ref{main}]
Let $C$ be the rational normal curve of degree $e$ on $X$. First, we show that $D_i$ is $N_{C|X}^*$-surjective. Since $d_i \geq 3$, by Theorem \ref{rathmann}, $d_iH$ is $N_{C|\PP^n}^*$-surjective.
The sheaf $N^*_{X\vert \PP^n}(d_iH)$ is  globally generated, since it is the quotient of the globally generated sheaf $\OI_{X\vert \PP^N}(d_iH)$. Hence, $H^1(C, N^*_{X\vert \PP^n}(d_iH)\vert_C)=0$. By Lemma \ref{lemma1},  we conclude that $d_iH$ is $N_{C|X}^*$-surjective. Lemma \ref{lemma2} then implies that $D_i=d_iH+E_i$ is $N_{C|X}^*$-surjective. The theorem is now a consequence of Proposition \ref{specifiedNormalBundle}.
\end{proof}
\subsection{Proof of Theorem \ref{mainApplied}}
 We now prove Theorem \ref{mainApplied} using Theorem \ref{main}. The proof requires a couple lemmas.

 \begin{lemma}\label{globallyGeneratedKernel}
Let $E$ and $F$ be globally generated vector bundles on $\PP^1$. Assume that $\shom(E,F)$ is  globally generated. If $\mathrm{rk}(E)> \mathrm{rk}(F)$ and $\deg(E)\geq \deg(F)$, then the kernel of a general map $E\ra F$ is globally generated.
\end{lemma}
\begin{proof}
We first prove the result if $F$ is a line bundle $\OO(b)$. By the hypotheses, we have $E\cong \bigoplus_{1\leq i \leq r} \OO(a_i)$ with $0\leq a_i\leq b$ for every $i$, and $a_1+\cdots+a_r\geq b$.
Reorder the $a_i$ such that 
\[0=a_1=\ldots=a_{r'}<a_{r'+1}\leq \ldots\leq a_r.\] 
Given $i$ with $r'\leq i\leq r$, define
\[
A_i=\twopartdef{0}{i=r'}{\min(b,a_{r'+1}+\cdots+a_i)}{i>r'}.
\]

Define the map $\phi:E\ra F$ by the matrix
\[
\phi=\begin{bmatrix}
0&\ldots&0&s^{b-A_{r'+1}}&s^{b-A_{r'+2}}t^{A_{r'+1}}&\ldots&s^{b-A_{i}}t^{A_{i-1}}&\ldots&s^{b-A_r}t^{A_{r-1}}
\end{bmatrix}.
\]
Then $\phi$ is surjective, and it induces a surjection $H^0(E(-1))\ra H^0(F(-1))$. To see this, consider the global section $s^{b-j-1}t^j\in H^0(F(-1))$. There is some $r'+1\leq i\leq r$ such that $A_{i-1}\leq j< A_i$. Then  $s^{b-j-1}t^j$ is the image of the section $s^{A_i-j-1}t^{j-A_{i-1}}\in H^0(\OO(a_i-1))$ under $\phi$.
If $K=\ker(\phi)$, then $H^1(K(-1))=0$, so $K$ is globally generated. Hence, if $E\ra F$ is general, its kernel will also be globally generated.

We now proceed by induction on the rank of $F$.
We have proven the result if $\rk(F)=1$.

Suppose the result has been proven for all $F$ of rank at most $s-1$. We establish the result for $F$ of rank $s$.
Let $F\cong \bigoplus_{1\leq j\leq s} \OO(b_j)$ and $E\cong \bigoplus_{1\leq i\leq r}\OO(a_i)$  with
\[
b_s\geq \ldots\geq b_1\geq a_r\geq \ldots\geq a_1
\]
and $\sum_i a_i\geq \sum_j b_j$.  Let $f:E\ra F$ be general, and set $F'=\bigoplus_{1\leq j\leq s-1} \OO(b_j)$
By the inductive hypothesis, the restricted map $E\ra F'$ has globally generated kernel $K$. 
Moreover, the map $\mathrm{Hom}(E,\OO(b_s))\ra \mathrm{Hom}(K,\OO(b_s))$ is surjective since $\mathrm{Ext}^1(F',\OO(b_s))=0$. 
So the induced map $f:K\ra \OO(b_s)$ is also general, and has globally generated kernel by the result for line bundles.
\end{proof}
The next lemma is a variant of \cite[Theorem IV.3.11]{Kol96}.
\begin{lemma}\label{generalization}
Let $g:X\ra S$ be a flat morphism with a connected base $S$.
Suppose that for some $s\in S$, $X_s$ has an immersion $i_s:\PP^1\ra X_s^{sm}$ such that $i^*TX_s$ is ample. 
Then there is an open subset $S'\subset S$ such that for any $s'\in S'$, there is an immersion $i_{s'}:\PP^1\ra X_{s'}^{sm}$ where $i^*TX_{s'}$ is ample.
\end{lemma}
\begin{proof}
Let $U\subset X$ be the open subset of $X$ on which $g$ is smooth. 
The set $U$ includes the image of the immersion $i_s$.
By the proof of \cite[Theorem IV.3.11]{Kol96}, there is an open subset $V$ of the Hom-scheme $\Hom_S(\PP^1_S,U)$ including $i_s$ such that for any $i'\in V$ the pullback $i'^*(T_{U/S})$ is ample.
Moreover, the projection $\Hom_S(\PP^1_S,U)\ra S$ is smooth in a neighborhood of $i_s$.
Hence, the image of $V$ in $S$ contains an open set $S'\ni s$ with the desired properties.
\end{proof}
 
\begin{proof}[Proof of Theorem \ref{mainApplied}]
Let $X\subset \PP^n$, $D_1,\ldots, D_c$, and $C$ be as in the statement of the theorem.
By Theorem \ref{main}, given a general surjection $q:N_{C\vert X}\ra \bigoplus_{1\leq i\leq c}\OO(D_i)\vert_C$, there exist hypersurfaces $\tilde{Y}_1,\ldots,\tilde{Y}_c$ with  $[\tilde{Y}_i]=D_i$ such that $\tilde{Y}=\bigcap_{1\leq i\leq c}\tilde{Y}_i$ is smooth along $C$ and $N_{C\vert \tilde{Y}}\cong \ker(q)$. 
Since $C$ is a rational normal curve in projective space, we have that every direct summand of $N_{C\vert X}$ has degree at most $\deg(C)+2$ by\cite[Corollary 2.6]{CR19}, so the bundle $N^*_{C\vert X}\otimes \bigoplus_{1\leq i\leq c}\OO(D_i)\vert_C$ is globally generated.
In addition, $q$ is general and $\deg(N_{C\vert \tilde{Y}})\geq m-c-1$ by hypothesis. So $N_{C\vert \tilde{Y}}$ is ample by Lemma \ref{globallyGeneratedKernel} applied to the morphism $N_{C\vert X}(-1)\ra \bigoplus_{1\leq i\leq c}\OO(D_i)\vert_C(-1)$.

Since smoothness and ampleness are open in families, if $Y_1,\ldots,Y_c$ are general hypersurfaces containing $C$, then $Y=\bigcap_{1\leq i \leq c} Y_i$ is smooth along $C$ and $N_{C\vert Y}$ is ample.

Moreover, $Y$ is an irreducible  variety of dimension $m-c$, as we now show.
Let $\tilde{X}=\Bl_C(X)$ with exceptional divisor $E$. The divisor $3H-E$ is very ample on $\tilde{X}$, as the restriction of the analogous very ample divisor $3H-E$ on $\Bl_C(\PP^n)$ to $\tilde{X}$. Since $(d_i-3)H+E_i$ is base-point free, the divisor classes $D_i-E$ on $\tilde{X}$ are also very ample. So, by the Bertini irreducibility theorem \cite[Theorem 1.1]{Ben11}, the complete intersection of $c\leq m-2$ general hypersurfaces of classes $D_1-E, \ldots,D_c-e$ on $\tilde{X}$ is irreducible. So $Y$ is likewise irreducible of dimension $m-c$

So there is a very free rational curve in the smooth locus of $Y$, and every component of $Y$ has dimension $m-c$ .
If $U$ is the family of all complete intersections of hypersurfaces of classes $D_1,\ldots,D_c$, and $\pi:\mathcal{Y}\ra U$ the universal complete intersection, then $\pi$ is flat by \cite[Theorem 23.1]{Mat86}.
So by Lemma \ref{generalization}, the general complete intersection $Y_u$ in this family contains a very free rational curve in its smooth locus, and is hence SRC.
\end{proof}

%% file: examples.tex
\section{Applications of Theorem \ref{mainApplied}}\label{sec-examples}
In this section, we use Theorem \ref{mainApplied} to prove that certain types of varieties are SRC. To apply the theorem to complete intersections on a particular variety $X$, we need to find a very free rational curve on $X$, linearly normal with respect to the ample class on $X$, that has sufficiently large degree that its restriction to complete intersections could be very free. 
 We handle the two cases of Theorem \ref{someMoreSRCVarieties} as the following two lemmas.

 Let $\Gr(k,n)$ denote the Grassmannian parameterizing $k$-dimensional subspaces of an $n$-dimensional vector space $V$. More generally,  for a sequence of nonnegative integers $0\leq k_1<k_2<\cdots<k_r\leq n$, the \emph{partial flag variety} $\Fl(k_1,\ldots,k_r;n)$ is  the parameter space of all partial flags
 \[
 0\subseteq V_{k_1}\subset V_{k_2}\subset\cdots\subset V_{k_r}\subseteq V, 
 \]
 where each $V_{k_i}$ is a $k_i$-dimensional vector space. 
\begin{lemma}\label{fl}
Let $X=\Fl(k_1,\ldots,k_r;n)$ be a flag variety. 
Let $X\ra \Gr(k_1,n)\times \cdots \times \Gr(k_r,n)$ be the canonical embedding, and let $H$ be the sum of the pullbacks of the hyperplane classes on each $\Gr(k_i,n)$. For each $i$ with $1\leq i \leq c$, let $D_i$ be a divisor class on $X$ such that $D_i-3H$ is effective. Let $Y$ be the general complete intersection of $c$ hypersurfaces of class $D_1,\ldots, D_c$. If $-K_X-D_1-\cdots-D_c$ is ample on $X$, then $Y$ is SRC.
\end{lemma}
\begin{proof}
The complete linear series $\lvert H\rvert$ gives an embedding of $X$ into projective space $\PP^N$, whose image is cut out by quadrics in $\PP^N$ by \cite[Theorem 3.11]{Ram87}.
Hence, to apply Theorem \ref{mainApplied}, we must find a very free rational normal curve $C$ on $X$ of large $H$-degree such that $H^0(X,H)\ra H^0(C, H\vert_C)$ is surjective.

Fix a basis $e_1,\ldots,e_n$ of the $n$ dimensional vector space $V$ in which $X$ parameterizes flags.
We start by constructing a rational curve on the Grassmannian $Gr(k_r,n)$ with nice properties.
Let $i:\PP^1\ra Gr(k_r,n)$ send a point $(s,t)$ to the span $V_{k_r}(s,t)$ of the $k_r$ vectors
\[
v_{k_r,j}(s,t)=\sum_{i=0}^{n-k_r} s^{n-k_r-i}t^i e_{i+j}
\]
with $1\leq j\leq k_r$. We note two properties of the image curve $C=i(\PP^1)$:
\begin{itemize}
    \item The restriction of the universal subbundle of $Gr(k_r,n)$ to $C$ is anti-ample; indeed, it is isomorphic to $\OO(k_r-n)^{\oplus k_r}$.
    \item If $Gr(k_r, n)$ is embedded in projective space by the Pl\"ucker embedding, then the image of $C$ is a rational normal curve. For if $0\leq a\leq k_r(n-k_r)$ is an integer satisfying $a=b(n-k_r)+c$ with $b,c$ integers and $0\leq c< n-k_r$, the monomial $s^a t^{k_r(n-k_r)}$ is expressible as the restriction of the Pl\"ucker coordinate $X_I$ to $C$, where $I$ is the set of $k_r$ coordinates
    \[
    I=\{1,2,\ldots, b, n-k_r+b+1-c,n-k_r+b+2,n-k_r+b+1,\ldots, n-1, n\}. 
    \]
\end{itemize}
We now extend this map to a map $i:\PP^1\ra X$ that will retain these two properties. We define the maps $i:\PP^1\ra Gr(k_i,V)$ inductively, downward from $r$.
First define the map $i:\PP^1\ra \Gr(k_r,n)$ as above; the above construction also gives a basis $v_{k_r,1}(s,t),\ldots,v_{k_r,k_r}(s,t)$ for each vector space $V_{k_r}(s,t)$.
Now suppose that we have a map $i:\PP^1\ra \Gr(k_{i+1},n)$ and a basis $v_{k_i{+1}, 1}(s,t),\ldots, v_{k_{i+1}, k_{i+1}}(s,t)$ for the spaces $V_{k_{i+1}}(s,t)$. For each $(s,t)\in \PP^1$, define $V_{k_i}(s,t)$ as the span of the vectors $\{v_{k_i,j}\vert 1\leq j\leq k_i\}$, with
\[
v_{k_i, j}(s,t)=s^{k_{i+1}-k_i} v_{k_{i+1},j}(s,t,)+t^{k_{i+1}-k_i}v_{k_{i+1},j+k_{i+1}-k_i}.
\]
This gives a map $i:\PP^1\ra \Gr(k_i,n)$. Repeating this process, we get a map $i:\PP^1\ra \Gr(k_i,V)$ for every $i$.

Set $\delta_i=\min\{ k_{j+1}-k_j\vert i\leq j\leq r-1\}$.
If we write $v_{k_i,j}$ in the form
\[
v_{k_i,j}(s,t)=\sum_{\ell=0}^{k_i} c_\ell s^{k_i-\ell}t^\ell e_{j+\ell},
\]
then, inductively, we have 
\[
c_0=\ldots c_{k_i}=1.
\]
As a result, the Pl\"ucker coordinates on $V_{k_i}$, restricted to $C$, include the monomials
\[
s^{k_i(n-k_i)}, s^{(k_i-1)(n-k_i)}t^{n-k_i},\ldots, t^{k_i(n-k_i)}.
\]
Since these include both the top power of $s$ and of $t$, $V_{k_i}(s,t)$ is everywhere $k_i$-dimensional. By construction, the vector spaces $V_{k_i}(s,t)$ fit into a chain of subspaces
\[
\{0\}\subset V_{k_1}(s,t)\subset \cdots \subset V_{k,r}(s,t)\subset V.
\]
Hence,  we have a map $i:\PP^1\ra X$. The image $C$ of $i$ has degree $d=\sum_{1\leq i \leq r} k_i(n-k_i)$ with respect to $H$. Moreover, any monomial $s^a t^{d-a}$ can be realized as a product of the Pl\"ucker coordinates on each Grassmannian factor.
Hence, the image of $C$ under the Pl\"ucker embedding of $X$ is a rational normal curve. Finally, $TX\vert_C$ is a quotient of $\shom(S_{k_r}, Q_{k_1})$, where $S_{k_r}$ is the universal sub-bundle on $Gr(k_r,n)$ and $Q_{k_1}$ the universal quotient bundle on $Gr(k_1,n)$. Since $S_{k_r}$ is anti-ample, $TX\vert_C$ is ample.
If $Y$ is a codimension $c$ Fano complete intersection satisfying the hypotheses of the lemma containing $C\subset X$, we have that $-K_Y$ is the restriction of an ample divisor on $X$. In particular, the difference $-K_Y-H$ is nef. As a result, since $C\cdot H=\sum_{1\leq i \leq r}k_i(n-k_i)$, we have
\[
-K_Y\cdot C\geq \sum_{i=1}^r k_i(n-k_i).
\]
Then, by Theorem \ref{mainApplied}, the general complete intersection in the class of $Y$ is SRC and in particular contains a very free curve that is a deformation of $C$.
\end{proof}
\begin{remark}
Even in situations where $Y$ is not Fano, the proof of Lemma \ref{fl}   produces rational curves with calculable normal bundles on $Y$.
For instance, if $(-K_Y-D_1-\cdots-D_c)\cdot C\geq 0$, where $C$ is the rational curve constructed in the proof, then $C$ will have a globally generated vector bundle in the general complete intersection containing it, so we can conclude that the general complete intersection $Y$ is separably uniruled.
\end{remark}
\begin{lemma}\label{pr}
Let $X$ be a product of projective spaces. For each $1\leq j\leq c$, let $D_i$ be a a divisor class of degree at least 3 on each factor space. Let $Y$ be the general complete intersection of $c$ hypersurfaces of classes $D_1,\ldots,D_C$. If $-K_X-D_1-\cdots-D_c$ is ample, then $Y$ is SRC.
\end{lemma}
\begin{proof}
Let $X=\PP^{a_1}\times \cdots \times \PP^{a_n}$, and let $C$ be a rational curve embedded in $X$ as a rational normal curve of degree $a_j$ in each factor $\PP^{a_j}$. Let $H_j$ be the pullback to $X$ of the hyperplane class on $\PP^{a_j}$, and set $H=H_1+\cdots+H_j$; $H$ is the minimal ample class on $X$. The curve $C$ is very free in $X$, and is linearly normal under the embedding of $X$ in projective space by $H$. If $c=0$, then $Y=X$ and $C$ being very free shows $X$ is SRC.

Now suppose that $Y$ is a Fano complete intersection containing $C$ satisfying the conditions of the lemma and of codimension $c\geq 1$. We have that $-K_Y\vert_C$ has degree at least $H\cdot C$, since $H$ is the minimal ample class on $X$. But $H\cdot C=a_1+\cdots+a_n=\dim X$, so the inequality
\[
-K_Y\cdot C \geq \dim (X)-c+1
\]
holds. Then by Theorem \ref{mainApplied}, $Y$ is SRC.
\end{proof}
Lemmas \ref{fl} and \ref{pr} collectively imply Theorem \ref{someMoreSRCVarieties}. Theorem \ref{grTheorem} is a special case of Lemma \ref{fl}.

Theorem \ref{mainApplied} certainly applies in some other instances; as the following results illustrate, in many cases the only real obstacle to applying it is finding a well behaved rational curve in the starting variety $X$. For instance, we have the following statement about Schubert varieties in homoegeneous varieties.
\begin{theorem}\label{schubertVarieties}
Let $V\subset \PP^n$ be a linearly normal homogeneous variety defined over an algebraically closed field $k$, let $X\subseteq V$ be a Schubert variety of dimension $m$, and suppose that $X$ has a very free rational curve $C$ contained in the smooth locus of $Y$ that is also a rational normal curve in $\PP^n$ of degree $e$. Let $\{d_i\}_{1\leq i\leq c}$ be a collection of integers, each at least 3, and suppose $-K_Y \cdot C\geq m+1-c+e(\sum_i d_i)$. Then a complete intersection of $Y$ with $c$ general hypersurfaces $D_i$ each of degree $d_i$ is SRC. 
\end{theorem}

\begin{proof}
 By \cite[Theorem 3.11]{Ram87}, X is cut out in $\PP^n$ by linear and quadric hypersurfaces. And, given any very free rational curve $C$ on $X$ that is a rational normal curve in $\PP^n$, and $\{d_i\}_{1\leq i \leq c}$ satisfying the inequality of the hypothesis, we have 
\[
C \cdot (-K_X-\sum_{1\leq i \leq c} D_i)\geq m-c+1.
\]
Then by Theorem \ref{mainApplied}, if $D_1$, \ldots, $D_c$ are general hypersurfaces of degrees $d_1,\ldots,d_c$, and $Y=D_1\cap \cdots \cap D_c \cap X$ is smooth, we have that $Y$ is SRC.

It remains to show that a general complete intersection contains a very free curve that is a deformation of $C$. In deformation theoretic terms, we want the map
\[
H^0(N_{C\vert \PP^n})\ra H^0(\bigoplus_{1\leq i \leq c} \OO(d_i H)\vert_C)
\]
to be surjective. This surjectivity follows from the very freeness of $C$ in the complete intersection.

\end{proof}
Likewise, Theorem \ref{mainApplied} applies to some weighted projective spaces.

\begin{lemma}\label{wps}
Let $X$ be the well-formed weighted projective space $\PP(a_0,\ldots,a_m)$.
Set $a'=\mathrm{lcm}(a_1,\ldots,a_m)$, and let $a=a'r$ be an integer such that $\OO(a)$ is very ample on $X$ and $X$ is cut out in $\PP(H^0(\OO(a))^*)$ by quadrics
Suppose that for each $0\leq i\leq m$ there exists an integer $0\leq b_i\leq ma_i$ such that each integer $0\leq \ell \leq ma$ can be expressed as a sum $\ell=c_0b_0+\cdots +c_mb_m$ with each $c_i$ a nonnegative integer and $ma=c_0a_0+\cdots+c_ma_m$.
Let $Y$ be a smooth Fano complete intersection of general hypersurfaces each of (weighted) degree at least $3a$ in $X$. Then $Y$ is SRC.
\end{lemma}
\begin{proof}
The result is trivial if $X$ is a weighted projective surface, since $X$ is SRC as a rational variety and any smooth Fano curve on it is isomorphic to $\PP^1$. 
Therefore, we may assume $m\geq 3$.
Let $b_i$ satisfy the hypotheses of the lemma.

Let $i':\PP^1\ra X$ be a map given by via the formula
$$i'(s,t)= (s^{b_0}t^{ma_0-b_0},\ldots, s^{b_m}t^{ma_m-b_m}).$$
 Then, by the hypothesis on the $b_i$, the weighted degree $a$ polynomials on $X$ include monomials that restrict to every monomial $s^{ma},\ldots, t^{ma}$ on $\PP^1$; in particular, the composition map $i':\PP^1\ra \PP(H^0(X,\OO(a)))$ embeds $\PP^1$ as a rational normal curve.
 By upper semicontinuity of cohomology, if $i:\PP^1\ra X$ is given by a general $m+1$-tuple of polynomials,
 \[
 i(s,t)=(f_0(s,t),\ldots, f_m(s,t))
 \]
 with each $f_i$ a polynomial of degree $a_im$, then $i(\PP^1)$ will also produce a rational normal curve in $\PP(H^0(X,\OO(a)))$. 
 Since $i$ is general, no two functions $f_i$ vanish simultaneously, so by the well-formedness of $X$, $i$ maps $\PP^1$ into the smooth locus of $X$. So the curve $C=i(\PP^1)$ is a very free curve, because the restricted tangent bundle $TX\vert_C$ is a quotient bundle of the ample bundle $\bigoplus_{0\leq i \leq m} \OO(ma_i)$ on $C$.

 Regarding $X$ as a projective subvariety of $\PP^N=\PP(H^0(X,\OO(a)))$, let $Y_1,\ldots, Y_c$ be general hypersurfaces containing $C$ with each $D_i$ of degree $d_i\geq 3$, and let $Y=Y_1\cap\cdots \cap Y_c$.
 If $Y$ is smooth and Fano, we have $a_0+\cdots+a_m-a(d_1+\cdots +d_c)>0$, whence 
 \[
 -K_Y \cdot C=m(a_0+\cdots+a_m-a(d_1+\cdots +d_c))\geq m
 \]
By Theorem \ref{mainApplied}, if $Y_i$ are general hypersurfaces of degree $d_i$ in $\PP(H^0(X,\OO(a)))$, then their complete intersection
is SRC provided it is smooth and Fano.
\end{proof}
\begin{example}
This lemma applies to the weighted projective space $\PP(1,\ldots,1,a_m)$ if $m\geq 2$. If $a=a_m$, $\OO(a)$ is very ample and its image is cut out by quadrics. And $b_{m-1}=m$, $b_m=ma_m$, and $b_i=i$ otherwise verifies the additional combinatorial hypothesis of the result. 
\end{example}

%% file: products.tex
\section{Products}\label{products}
In this section, we discuss the normal bundles of rational curves in products of varieties.

Products of SRC varieties are SRC as the following argument shows. Let $X_1$ and $X_2$ be SRC varieties and let $C_i$ be a very free curve in $X_i$.  Let $C$ be an immersed $(1,1)$ curve in $C_1\times C_2$. Then we have the exact sequence
\[
0\ra N_{C\vert C_1\times C_2}\ra N_{C\vert X_1\times X_2}\ra N_{C_1\times C_2\vert X_1\times X_2}\vert_C\ra 0.
\]
Since the first and last bundles in this sequence are ample, the one in the middle is as well. 
In general, the normal bundles of the two projections of $C$ in $X_1$ and $X_2$ do not determine the normal bundle of $C$ in $X_1 \times X_2$. However, if $C$ is general, Theorem \ref{productsMain} asserts that $N_{C\vert X_1\times X_2}$ is a general quotient of the restricted tangent bundle $T(X_1\times X_2)\vert_C$ by $TC$.

We will now prove Theorem \ref{productsMain}. Let $f:\PP^1\ra X_1\times X_2$ be an immersion. For any integer $d$, we have an exact sequence
\[
0\ra N^*_f(d)\ra f^*(T^*X_1\oplus T^*X_2)(d)\ra \OO(d-2)\ra 0.
\]
Let $V_{i,d}\subset H^0(\PP^1,\OO(d-2))$ be the image of $H^0(\PP^1,f^*(T^*X_i)(d))$ in the associated long exact sequence. Let $v_{i,d}$ denote the dimension of $V_{i,d}$. If $V_{1,d}$ and $V_{2,d}$ are transverse, then the image of $H^0(\PP^1, f^*(T^*X_1)(d)\oplus f^*(T^*X_2)(d))$ in $H^0(\PP^1, \OO(d-2))$ has dimension $\min(d-1, v_{1,d}+v_{2,d})$.
In general, $V_{1,d}$ and $V_{2,d}$ do not have to be transverse. For example, if $X_1=X_2$ and $f_1=f_2$, then $V_{1,d}=V_{2,d}$.
However, if $f$ is general and $k$ is a field of sufficiently large characteristic, the following proposition guarantees that $V_{1,d}$ and $V_{2,d}$ are transverse.
Set
\[
d_0=\min\{d\geq 2\,\vert\, V_{1,d}=H^0(\PP^1,\OO(d-2)) \;\mathrm{or}\; V_{2,d}=H^0(\PP^1,\OO(d-2)) \}.
\]
Serre vanishing guarantees the existence of this $d_0$.
\begin{proposition}\label{productsProp}
Assume the characteristic of the base field is 0 or $p\geq d_0-1$. Let the maps $f_1:\PP^1\ra X_1$ and $f_2:\PP^1\ra X_2$ be immersions, and let $\alpha:\PP^1\ra \PP^1$ be a general automorphism. Let $f:\PP^1\ra X_1\times X_2$ be given by $(f_1,f_2\circ \alpha)$.
 Then, for any $d$, $V_{1,d}$ and $V_{2,d}$ are transverse.
\end{proposition}
\begin{proof}
Fix an integer $d$. If $v_{1,d}=d-1$ or $v_{2,d}=d-1$, the result is trivial, so we assume both $v_{1,d}<d-1$ and $v_{2,d}<d-1$. The maps $f_1$ and $f_2$ induce maps $$Df_i:H^0(\PP^1, f_i^*T^*X_i(d))\ra H^0(\PP^1, T^*\PP^1(d)).$$ Set $V=H^0(\PP^1, T^*\PP^1(d))$ and let $V_1$ and $V_2$ be the images of $f_1$ and $f_2$ respectively. There is a natural action of $SL_2$ on $H^0(\PP^1,T^*\PP^1(d))$ such that $\alpha\in SL_2$ acts by composition  $\alpha \eta=\eta \circ \alpha$. Then the image of $H^0(\PP^1, (f^*T^*(X_1\times X_2))(d))$ in $V$ is the span of $V_1$ and $\alpha V_2$.  We want to show that for general $\alpha$  these two vector spaces are transverse. 

Suppose first that $v_{1,d}+v_{2,d}\leq d-1$. Using an isomorphism $T^*\PP^1\cong \OO(-2)$, we fix an identification $H^0(\PP^1,T^*\PP^1(d))$ with $H^0(\PP^1,\OO(d-2))$ as an $SL_2$ representation for the remainder of the proof. Let $g_1,\ldots, g_{v_{2,d}}$ be a basis for $V_2$, so each $g_i$ is a degree $d-2$ homogeneous polynomial in coordinates $s,t$ on $\PP^1$. Set $v_2:=g_1\wedge \cdots \wedge g_{v_{2,d}}\in \Lambda^{v_{2,d}} V$. By \cite[Lemma 2.3]{Ya96} and the bound on $p$,
 we have that the Wronskian
 \[
 \begin{vmatrix}
 g_1&\frac{\partial }{\partial t}g_1&\dots&\frac{\partial^{v_{2,d}-1}}{\partial t^{v_{2,d}-1}}g_1\\
 \vdots &\vdots&\ddots&\vdots\\
  g_{v_{2,d}}&\frac{\partial }{\partial t}g_{v_{2,d}}&\dots&\frac{\partial^{v_{2,d}-1}}{\partial t^{v_{2,d}-1}}g_{v_{2,d}}
 \end{vmatrix}
 \]
 does not vanish everywhere. Pick coordinates $s,t$ on $\PP^1$ such that the Wronskian does not vanish at $t=0$. 
Let $U$ be the smallest subrepresentation of $\Lambda^{v_2}(V)$ containing $v_2$.

For any $a$ let $V_a$ be the $SL_2$ representation $\sym^aV_1$, where $V_1$ is the natural 2-dimensional representation of $SL_2$.
The Wronskian map is a map of representations $\Lambda^{v_{2,d}}V\ra V_b$, where $b=v_{2,d}\frac{2d-3-v_{2,d}}{2}$. Then $V_b$ is the highest weight direct summand of $\Lambda^{v_{2,d}}V\ra V_b$, so since $U$ is not in the kernel of the map $\Lambda^{v_{2,d}}V\ra V_b$, $U$ must contain $V_b\subset \Lambda^{v_{2,d}}V$.
In particular, the element of $V_b\subset \Lambda^{v_{2,d}}V$,  $v_2'=s^{d-2}\wedge\cdots\wedge s^{d-1-v_{2,d}}t^{v_{2,d}-1}$ is in the span of $v_2$ under the $SL_2$ action.
 
 By a similar argument, if we set $v_1$ as the element of $\Lambda^{v_{1,d}}V$ corresponding to $V_1$, we have that $v_1':=t^{d-2}\wedge \cdots \wedge s^{v_{1,d}-1}t^{d-1-v_{1,d}}$ is in the $SL_2$ span of $v_1$. As a consequence, because $v_1'\wedge v_2'$ is nonzero, for some $\alpha\in SL_2$ we have $v_1\wedge \alpha v_2\neq 0$. Then $V_1$ and $\alpha V_2$ have zero intersection.
 
 If $v_{1,d}+v_{2,d}\geq d-1$, let $V'_1\subset V_1$ be a subspace of complementary dimension to $V_2$. Then the argument above establishes that for some $\alpha$, $V'_1$ and $\alpha V_2$ span $V$, giving the desired result.
 \end{proof}

\begin{proof}[Proof of Theorem \ref{productsMain}]
Let the maps $f_i:\PP^1\ra X$ be as in the statement of the result.
Suppose $\mathrm{char}(k)=0$ or $\mathrm{char}(k)=p$ and $H^0(\PP^1,f_1^*(T^*X_1)(p+2))\ra H^0(\PP^1,T^*\PP^1(p+2))$ is surjective; by hypothesis, this can always be accomplished by permuting the indices.
In what follows, we assume the characteristic is positive, because the characteristic zero case merely requires removing all reference to the characteristic.
Since the product maps
\[
H^0(\PP^1,T^*\PP^1(p+2))\otimes H^0(\PP^1,\OO(a))\ra H^0(\PP^1, T^*\PP^1(p+a+2)
\]
are surjective for all $a\geq 0$, the above hypothesis implies that for all $d\geq p+2$ the maps $H^0(\PP^1, f_1^*(T^*X_1)(d)\ra H^0(\PP^1,T^*\PP^1(d))$ are surjective. 

Set $Y_i = X_1 \times \cdots \times X_i$. Define $g_i:\PP^1\ra Y_i$ inductively by $g_1=f_1$  and $g_i=(g_{i-1}, f_i\circ \alpha_i )$, where $\alpha_i$ is a general automorphism of $\PP^1$.
Since $\alpha_i^*(E)\cong E$ for any vector bundle $E$ on $\PP^1$, we have that $g_i^*(TY_i)$ is isomorphic to $(f_1,\ldots,f_i)^*(TY_i)$.
In addition, because the map 
\[
H^0(\PP^1,f_1^*(T^*X_1)(d))\ra H^0(\PP^1,T^*\PP^1(d))
\]
is surjective if $d\geq p+2$, the maps
\[
H^0(\PP^1,g_i^*(T^*Y_i)(d))\ra H^0(\PP^1,T^*\PP^1(d))
\]
are surjective as well, since $g_i^*(T^*Y_i)(d)$ has $f_1^*(T^*X_1)(d)$ as a quotient.

So if $d\geq p+2$, we have
\[
h^0(\PP^1,N^*_{g_i}(d))=h^0(\PP^1, g_i^*(T^*Y_i)(d))-d+1
\]
automatically, verifying the result.
Finally, if $d\leq p+1$, applying Proposition \ref{productsProp} to the pair of morphisms $(g_{i-1},f_i)$, we have 
\[
h^0(\PP^1, N^*_{g_i}(d))=\max(h^0(\PP^1, g_i^*(T^*Y_i)(d))-d+1, h^0(\PP^1, N^*_{g_{i-1}}(d))+h^0(\PP^1, N^*_{f_i}(d)).
\]
Combining these formulas, noting $h^0(f_i^*(T^*X)(d))\geq h^0(N^*_{f_i}(d))$, and setting $g=g_r$, we have that $g^*(TX)\cong f^*(TX)$ and 
\[
h^0(\PP^1, N^*_g(d))=\max(h^0(\PP^1,f^*(T^*X)(d))-d+1, \sum_{i=1}^r h^0(\PP^1,N_{f_i}^*(d))).
\]
\end{proof}
 \begin{remark}
The restriction on the characteristic is needed in Proposition \ref{productsProp}---and hence in Theorem \ref{productsMain}. Let $\mathrm{char}(k)=p\geq 3$, and let $f:C\ra \PP^3$ be the rational curve embedded by the map $$(s,t)\mapsto (s^{p+1},s^{p}t, st^{p},t^{p+1}).$$ The restricted cotangent bundle $T^*\PP^3\vert_C$ is $\OO(-p-2)^{\oplus 2}\oplus \OO(-2p)$, and the induced map $T^*\PP^3\vert_C\ra T^*C$ is given by $(s^p,t^p,0)$. In particular, the image of the map on $H^0$ induced by this map does not change if $f$ is precomposed with an automorphism. The map $(f \circ \alpha, f):C\ra \PP^3\times \PP^3$ induced by twisting $f$ hence gives a map
 $H^0(\PP^1, (f^*T^*(\PP^3\times \PP^3))(d))\ra H^0(\PP^1, T^*\PP^1(d))$ with an image of dimension $\max(0, \min(2(d-p-1), d-1))$. Consequently, $N_{C\vert \PP^3\times \PP^3}$ has splitting type  $\OO(-2p)^{\oplus 2}\oplus \OO(-2p-2)\oplus \OO(-p-2)^{\oplus 2}$ instead of the general splitting type.

 \end{remark}

%% file: main.bbl
\begin{thebibliography}{ABCH}
\bibitem[AR17]{AlzatiRe}
A. Alzati and R. Re, Irreducible components of Hilbert Schemes of rational curves with given normal bundle, Algebr. Geom., {\bf 4} no. 1 (2017), 79--103.

\bibitem[Ben11]{Ben11}
O. Benoist, Le th{\'e}or{\`e}me de {B}ertini en famille, Bull. Soc. Math. France {\bf 139} no. 4 (2011), 555--569.

\bibitem[Br13]{Bridges}
T. Bridges, R. Datta, J. Eddy, M. Newman and J. Yu, Free and very free morphisms into a Fermat hypersurface, Involve {\bf 6} (2013), no. 4, 437--445.

\bibitem[CZ14]{CZ14}
Q. Chen and Y. Zhu, Very free curves on Fano complete intersections, Algebr. Geom., \textbf{1} no. 5 (2014) 558--572.

\bibitem[Con06]{Conduche}
 D. Conduch\'{e}, Courbes rationnelles et hypersurfaces de l'espace projectif, Ph.D. thesis, Universit\'{e}. Louis Pasteur, 2006.
 


\bibitem[CR18]{CoskunRiedl}
I. Coskun and E. Riedl, Normal bundles of rational curves in projective space, Math. Z. {\bf 288}  no. 3-4 (2018), 803--827. 

\bibitem[CR19]{CR19}
I. Coskun and E. Riedl, Normal bundles of rational curves on complete intersections, Commun. Contemp. Math. \textbf{21} no. 2 (2019), 29 pp.  


\bibitem[EV81]{EisenbudVandeven}
D. Eisenbud and A. Van de Ven, On the normal bundles of smooth rational space curves, Math. Ann., {\bf 256} (1981), 453--463.

\bibitem[EV82]{EisenbudVandeven2}
D. Eisenbud and A. Van de Ven, On the variety of smooth rational space curves with given degree and normal bundle, Invent. Math., {\bf 67} (1982), 89--100.


  

\bibitem[GS80]{GhioneSacchiero}
F. Ghione and G. Sacchiero, Normal bundles of rational curves in $\PP^3$, Manuscripta Math., {\bf 33} (1980), 111--128.







\bibitem[K96]{Kol96}
J. Koll\'{a}r, Rational curves on algebraic varieties, Springer, 1996.

  



\bibitem[LT19]{LT19}
B. Lehmann and S. Tanimoto, Rational curves on prime Fano threefolds of index 1, J.  Algebraic Geom., {\bf 30} (2021), 151--188.


\bibitem[Mat86]{Mat86}
H. Matsumura, Commutative ring theory, Cambridge University Press, 1986.

\bibitem[Rat16]{Rat16}
J. Rathmann, An infinitesimal approach to a conjecture of Eisenbud and Harris, preprint, arXiv:1604.06069.


\bibitem[Ram87]{Ram87}
A. Ramanathan, Equations defining Schubert varieties and Frobenius splittings of diagonals, Publ. Math. IH{\'E}S, {\bf 65} (1987), 61--90.
  

\bibitem[Ran07]{Ran}
Z. Ran, Normal bundles of rational curves in projective spaces, Asian J. Math. {\bf 11}  no. 4 (2007), 567--608.

\bibitem[Sa80]{Sacchiero2}
G. Sacchiero, Fibrati normali di curvi razionali dello spazio proiettivo, Ann. Univ. Ferrara Sez.
VII, {\bf 26} (1980), 33--40.

\bibitem[Sa82]{Sacchiero}
G. Sacchiero, On the varieties parameterizing rational space curves with fixed normal bundle, Manuscripta Math., {\bf 37} (1982), 217--228.

\bibitem[Sh12a]{She12}
M. Shen, Rational curves on Fermat hypersurfaces, C. R. Math. Acad. Sci. Paris {\bf 350} (2012), no.~15-16, 781--784.

\bibitem[Sh12b]{Shen2}
M. Shen, On the normal bundles of rational curves on Fano 3-folds, Asian J. Math. {\bf 16} no. 2 (2012), 237--270.

\bibitem[Ti15]{Tia15}
Z. Tian, Separable rational connectedness and stability, in Rational points, rational curves, and entire holomorphic curves on algebraic varieties, Contemp. Math., \textbf{654} (2015), 155--160.
\bibitem[Ya96]{Ya96}
J. T.-Y. Yang, The truncated second main theorem of function fields, J. Number Theory {\bf 58} (1996), 139--157.

\bibitem[Zh11]{Zhu}
Y. Zhu, Fano hypersurfaces in positive characteristic, preprint, arXiv:1111.2964.

\end{thebibliography}
